\documentclass[12pt]{article} 
\usepackage{graphicx}
\usepackage{amssymb}
\usepackage{mathtools}
\usepackage{amsmath}
\usepackage{amsthm}
\usepackage{color}
\usepackage[hidelinks]{hyperref}
\usepackage{csquotes}
\usepackage[pagewise]{lineno}
\usepackage[english]{babel}
\usepackage{enumitem}
\usepackage{graphicx,url} 
\usepackage{tocloft}
\usepackage{sectsty}
\usepackage{mathrsfs}
\usepackage{eufrak}
\usepackage{comment}
\usepackage{multirow}
\usepackage{booktabs}

\renewcommand{\Re}{\operatorname{Re}}
\newtheorem{thm}{Theorem}[section]

\newtheorem{lem}[thm]{Lemma}
\newtheorem{prop}[thm]{Proposition}
\newtheorem{defn}{Definition}[section]
\newtheorem{exmp}{Example}[section]
\newtheorem{rmk}{Remark}[section]

\numberwithin{equation}{section}

\renewcommand{\Re}{\operatorname{Re}}

\sectionfont{\fontsize{13}{15}\selectfont}
\subsectionfont{\fontsize{12}{15}\selectfont}

\def \R{{\mathbb{R}}}
\def \P {\Phi(x)}
\def \o {\omega(x)}

\def \japxik{\langle \xi \rangle_k}
\def \japx{\langle x \rangle}
\def \hyp{Z_{ext}(N)}
\def \pd{Z_{int}(N)}

\def \J{[0,T] \times \R^{2n}}
\def \la{\langle}
\def \ra{\rangle}
\def \ran{\rangle_k}
\def \rak{\rangle_k}

\def \wm{(\omega,g_{\Phi,k})}

\def \t{\theta(t)}
\def \g{\gamma}

\providecommand{\keywords}[1]
{
	\small	
	\textbf{\text{Keywords:}} #1
}

\providecommand{\subclass}[1]
{
	\small	
	\textbf{\text{MSC(2010):}} #1
}

\usepackage{geometry}
\title{\Large Global Well-Posedness of a Class of Hyperbolic Cauchy Problems with Coefficients Sublogarithmic in Time\thanks{Dedicated to Bhagawan Sri Sathya Sai Baba on the Occassion of His 96th Birthday.}}

\author{\normalsize Rahul Raju Pattar\thanks{rahulrajupattar@gmail.com (Corresponding Author)} , N. Uday Kiran\thanks{nudaykiran@sssihl.edu.in}  \\
	\small Department of Mathematics and Computer Science\\
	\small Sri Sathya Sai Institute of Higher Learning, Puttaparthi, India \\
}

\date{}

\begin{document}
	
	\maketitle
	
	\begin{abstract}
		The goal of this paper is to study global well-posedness, cone of dependence and loss of regularity of the solutions to a class of strictly hyperbolic equations with coefficients displaying \enquote{mild} blow-up of sublogarithmic order - $|\ln t|^\gamma,\gamma\in(0,1).$
		The problems we study are of strictly hyperbolic type with respect to a generic weight and a metric on the phase space.
		The coefficients are polynomially bound in $x$ with their $x$-derivatives and $t$-derivative of order $\textnormal{O}(t^{-\delta}),\delta \in [0,1),$ and $O(t^{-1}|\ln t|^{\gamma-1}), \gamma\in(0,1),$ respectively. 
		We employ the Planck function associated with the metric to subdivide the extended phase space and define appropriate generalized parameter dependent symbol classes.
		To arrive at an energy estimate, we perform a conjugation by a pseudodifferential operator. This operator explains the loss of regularity by linking it to the metric on the phase space and the singular behavior. We call the conjugating operator as {\itshape{loss operator}}.
		We report that the solution experiences an arbitrarily small loss in relation to the initial datum defined in the Sobolev space tailored to the loss operator.  \\
		\keywords{Loss of Regularity $\cdot$ Strictly Hyperbolic Operator with non-regular Coefficients $\cdot$ Sublogarithmic Singularity $\cdot$ Metric on the Phase Space $\cdot$ Global Well-posedness $\cdot$ Pseudodifferential Operators}\\
		\subclass{35L81 $\cdot$ 35L15 $\cdot$ 35B65 $\cdot$ 35B30 $\cdot$ 35S05 }
	\end{abstract}
	
	
	\section{Introduction}		
	We study the Cauchy problem for the strictly hyperbolic equation: 
	\begin{linenomath*}
	\begin{equation}\label{eq01}
		\partial_t^2u - \sum_{i,j=1}^{n} a_{i,j}(t,x)\partial_{x_i} \partial_{x_j}u + \sum_{j=1}^{n} b_{j}(t,x)\partial_{x_j}u + b_{n+1}(t,x)u=f(t,x),
	\end{equation}
	\end{linenomath*}
	with initial data $u(0,x)=f_1(x), \; \partial_tu(0,x)=f_2(x)$,  $(t,x)\in [0,T] \times \R^n.$
	Here $a_{i,j}\in C^1((0,T];C^\infty(\R^n))$ and $b_{j},b_{n+1}\in C([0,T];C^\infty(\R^n)), \; i,j=1,\dots,n,$ satisfy the following estimates
	\begin{linenomath*}
	\begin{align}
		\label{ell}
		a_{i,j}(t,x) & \geq C_{0} \o^2,\\ 
		\label{sing}
		|\partial_x^\beta\partial_t a_{i,j}(t,x)| & \leq C_{\beta} \frac{1}{t}  \left( \ln \left(1+\frac{1}{t}\right) \right)^{\gamma(1+|\beta|)-1}\o^2\P^{-|\beta|}\\
		|\partial^\beta_xb_j(t,x)| &\leq C_{\beta}' \; \o\P^{-|\beta|},\nonumber
	\end{align}
	\end{linenomath*}
	where $\gamma \in (0,1)$ and $\beta \in \mathbb{N}_0^n$, $C_{0},C_{\beta},C_{\beta}'>0$.
	Here the functions $\o$ and $\P$ are positive monotone increasing in $|x|$ such that $1\leq \o \lesssim \P\lesssim \japx = (1+|x|^2)^{1/2}.$ The functions $\omega$ and $\Phi$ specify the structure of the differential equation in the space variable. These functions will be discussed in detail in Section \ref{tools}.  
	
	From (\ref{sing}) with $|\beta| =0$, we have
	\begin{linenomath*}
		\[
		\vert a_{i,j}(T,x) - a_{i,j}(t,x) \vert \leq \int_t^T \vert \partial_s a_{i,j}(s,x) ds \vert \leq C \left(\ln \left(1+ \frac{1}{t} \right)\right)^{\gamma} \o^2.
		\]
	\end{linenomath*}
	Implying  
	\begin{equation}\label{logblow}
		\vert a_{i,j}(t,x) \vert \leq C \left( \ln  \left(1+\frac{1}{t} \right)\right)^{\gamma}\o^2,
	\end{equation}
	i.e., the coefficient $a(t,x)$ has a sublogarithmic blow-up at $t=0.$ 
	\begin{rmk}\label{rmk}
		We can also replace (\ref{sing}) with a slightly general estimates similar to the ones given in \cite{Cico1,RU2} by
		\begin{equation}\label{sing2}
			\begin{rcases}
				\begin{aligned}
					|\partial_x^\beta a_{i,j}(t,x)| & \leq C_{\beta} \frac{1}{t^{\delta_1}} \o^2\P^{-|\beta|}, \quad |\beta| >0,\\
					|\partial_x^\beta\partial_t a_{i,j}(t,x)| & \leq C_{\beta} \frac{(\ln(1+1/t))^{\gamma-1}}{t^{1+\delta_2|\beta|}} \o^2\P^{-|\beta|}, \quad |\beta| \geq 0,
				\end{aligned}
			\end{rcases} \tag{\ref{sing}${}^*$}
		\end{equation}
		for $\delta_1,\delta_2 \in [0,1).$
	\end{rmk}
	Examples of functions satisfying (\ref{ell}) and (\ref{sing2}) are given below.
	
	\begin{exmp}\label{BLOS2}
		Let $n=1$, $T=1$, $\g \in (0,1)$, $\kappa_1 \in[0,1]$ and $\kappa_2\in(0,1]$ such that $\kappa_1 \leq \kappa_2$. Then, 
		\begin{linenomath*}
			$$
			\japx^{2\kappa_1} \left(2+\sin \left( \japx^{1-\kappa_2} + \cos x \; \left(\ln(1+1/t)\right)^\gamma\right) + \left( 2+\cos  \japx^{1-\kappa_2} \right)\left(\ln(1+1/t)\right)^\gamma\right)
			$$
		\end{linenomath*}
		satisfies the estimates (\ref{sing2}) for $\o = \japx^{\kappa_1}$, $\P = \japx^{\kappa_2}$ and for any $\delta_1,\delta_2 \in (0,1)$.
	\end{exmp}
	\begin{figure}[h!]
		\centering
		\includegraphics[scale=0.65]{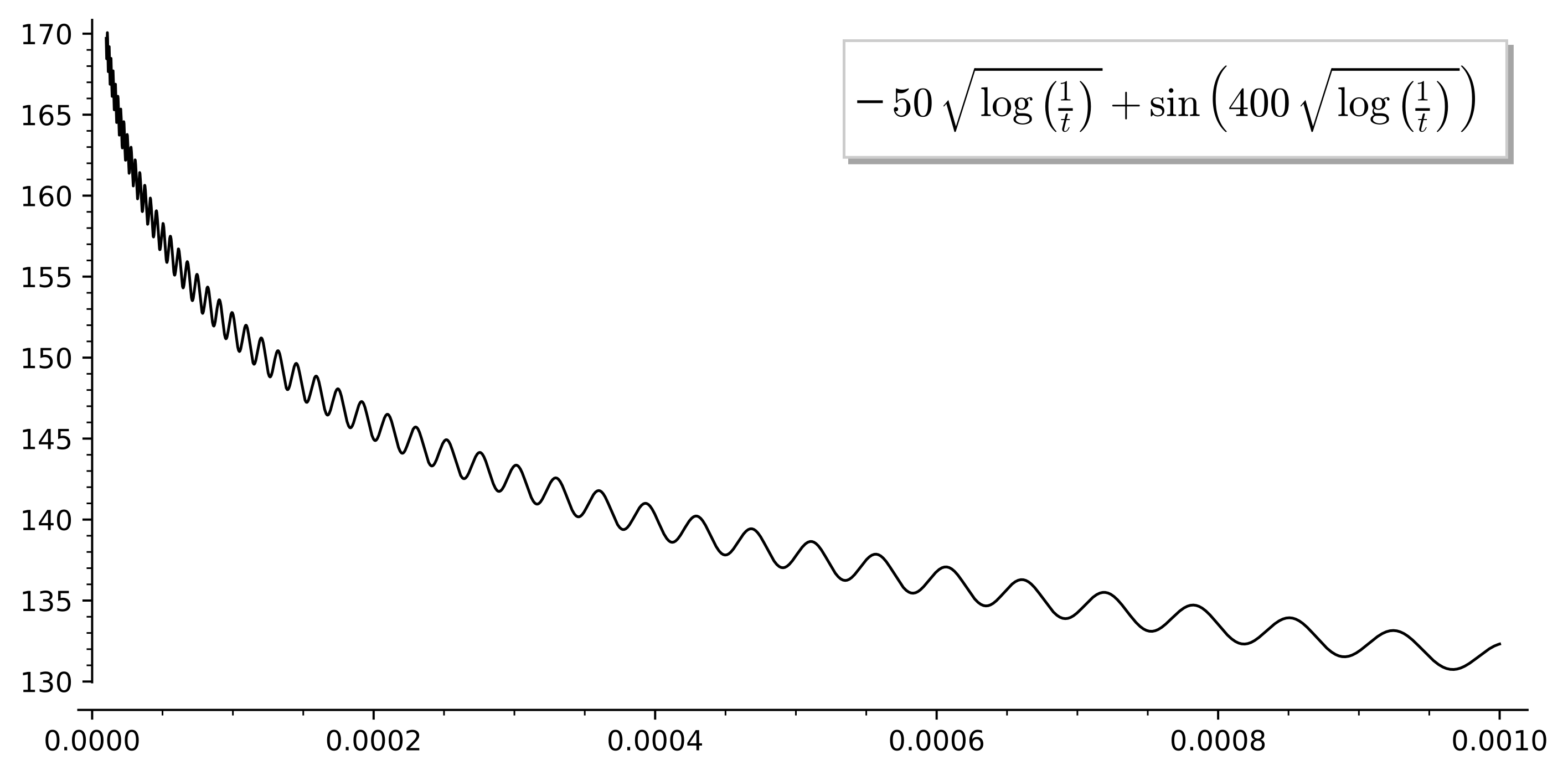}
		\caption{Graph of Example \ref{BLOS}.}
		\label{Msublog}
	\end{figure}
	\begin{exmp}\label{BLOS}
		Let $T=0.1$ and $n=1$. Then,
		\[
			50\sqrt{\ln(1/t)} + \sin(400\sqrt{\ln(1/t)})
		\]
		satisfies the estimates (\ref{sing2}) for $\gamma=\frac{1}{2}$, $\o = \P = 1$ and $\delta_1=\delta_2=0$.
	\end{exmp}

	As evident from (\ref{sing2}-\ref{logblow}) and Examples \ref{BLOS2}-\ref{BLOS}, we consider coefficients that are singular at $t=0$. The singular behavior is characterized by either sublogarithmic blow-up or infinitely many oscillations near $t=0.$
	Below is the scale for the rate of oscillations which is inspired from the one given in \cite{Reis}. 	
	\begin{defn}[Oscillatory Behavior]
		Let $c=c(t) \in L^{\infty}([0,T]) \cap C^2((0,T])$ satisfy the estimate
		\begin{equation}\label{oscill}
			\left\vert \frac{d^j}{dt^j} c(t) \right\vert \lesssim \left( \frac{\vert\ln t  \vert^{\gamma {\bf I}_q}}{t^q}  \right)^j,  
		\end{equation}
		for $j=1,2$ and $q \geq 1.$ The function ${\bf I}_q$ is such that ${\bf I}_q \equiv 1$ if $q=1$ else ${\bf I}_q \equiv 0.$
		We say that the oscillating behavior of the function $c$ is
		\begin{itemize}
			\item very slow if $q=1$ and $\gamma = 0$
			
			\item slow if $q=1$ and $\gamma  \in(0,1)$
			
			\item fast if $q=\gamma =1$
			
			\item very fast if $q>1$ or else $\gamma >1$ when $q=1$.
		\end{itemize} 
	\end{defn}
	\hspace{-0.6cm}When $q=1,$ note that the above definition matches the one in \cite{Reis}. We have redefined very fast oscillation by introducing the case of $q>1.$  Table  \ref{T1} summarizes the loss of regularity results in the case of oscillatory coefficients.
	
		\begin{table}[h!]
		\small\centering
		\begin{tabular}{|cc|c|cc|c|c|}
			\hline
			\multicolumn{2}{|c|}{\begin{tabular}[c]{@{}c@{}}Order of\\ Oscillations\end{tabular}} & \multirow{2}{*}{\begin{tabular}[c]{@{}c@{}}Regularity in $t$\\ of coefficients\\ \end{tabular}} & \multicolumn{2}{c|}{\begin{tabular}[c]{@{}c@{}}Growth in $x$\\ of coefficients\end{tabular}} & \multirow{2}{*}{\begin{tabular}[c]{@{}c@{}}Loss of regularity\\ index for solution\end{tabular}} & \multirow{2}{*}{Ref.} \\ \cline{1-2} \cline{4-5}
			
			\multicolumn{1}{|c|}{$q$}                            & $\gamma$                       &                                                                                                & \multicolumn{1}{c|}{$\omega$}                           & $\Phi$                             &                                                                                                  &                       \\ \hline
			
			\multicolumn{1}{|c|}{1}                              & $[0,1)$                        & $L^{\infty}([0,T])\cap C^2((0,T])$                                            & \multicolumn{1}{c|}{1}                          & 1                          & \begin{tabular}[c]{@{}c@{}}Zero to\\ arbitrarily small\end{tabular} & \cite{Reis}                  \\ \hline
			
			\multicolumn{1}{|c|}{1}                              &$(0,1)$                        &  $C^2((0,T])$                                                                 & \multicolumn{1}{c|}{ $\omega(x)$}                &  $\Phi(x)$                  & \begin{tabular}[c]{@{}c@{}} Zero to\\  arbitrarily small\end{tabular} &  \cite{RU3}                  \\ \hline
			
			\multicolumn{1}{|c|}{1}                              & (0,1]                              & $C([0,T])\cap C^{2}((0,T])$                                     & \multicolumn{1}{c|}{1}                          & 1                          & Finite                                                              & \cite{GG}                  \\ \hline
			
			\multicolumn{1}{|c|}{1}                              & 1                              & $L^{\infty}([0,T])\cap C^{\infty}((0,T])$                                     & \multicolumn{1}{c|}{1}                          & 1                          & Finite                                                              & \cite{KuboReis}                  \\ \hline
			
			\multicolumn{1}{|c|}{1}                              & 1                              & $L^{\infty}([0,T])\cap C^{\infty}((0,T])$                                     & \multicolumn{1}{c|}{$\langle x \rangle$}         & $\langle x \rangle$        & Finite                                                              & \cite{NUKCori1}                  \\ \hline
			
			\multicolumn{1}{|c|}{1}                              &  $[0,\infty)$                   &  $C^2((0,T])$                                                                 & \multicolumn{1}{c|}{ $\omega(x)$}                &  $\Phi(x)$                  & \begin{tabular}[c]{@{}c@{}} Zero to\\  Infinite\end{tabular}        &  \cite{RU3}                \\ \hline
			
			\multicolumn{1}{|c|}{$(1,\infty)$}                   & -                              & $C^1((0,T])$                                                                 & \multicolumn{1}{c|}{1}                          & 1                          & Infinite                                                            & \cite{Cico1}                  \\ \hline
			
			\multicolumn{1}{|c|}{ $\left(1,\frac{3}{2}\right)$}                      & -                              &  $C^1((0,T])$                                                                & \multicolumn{1}{c|}{ $\omega(x)$}                &  $\Phi(x)$                  &  Infinite                                                            &  \cite{RU1}             \\ \hline
		\end{tabular}
		\label{T1}
		\caption{ Loss of regularity in case of oscillatory coefficients.}
	\end{table}
	Following is a scale for the blow-up rate.
	\begin{defn}[Blow-up Rate]\label{BUdefn}
		Let $c=c(t) \in L^{1}((0,T]) \cap C^1((0,T])$ satisfy the estimates
		\begin{equation}\label{BU}
			\begin{rcases}
				\begin{aligned}
					| c(t)| &\lesssim \frac{1}{t^p} |\ln t|^{\gamma {\bf I}_q},\\
					| \partial_t c(t)| &\lesssim \frac{1}{t^q} |\ln t|^{(\gamma-1){\bf I}_q},
				\end{aligned}
			\end{rcases}
		\end{equation}
		with $q \in [1,\infty), p \in [0,1), p \leq q-1.$ The function ${\bf I}_q$ is such that ${\bf I}_q \equiv 1$ if $q=1$ else ${\bf I}_q \equiv 0.$ We say that the blow-up rate of the function $c$ is
		\begin{itemize}
			\item mild if $q = 1, p=0, \gamma \in (0,1)$
			
			\item logarithmic if $q = 1, p=0, \gamma=1$
			
			\item strong if $q = 1, p=0, \gamma \in (1,\infty)$
			
			\item very strong if $q>1, p\in[0,1).$ 
		\end{itemize} 
	\end{defn}	
	Table 2 summarizes the loss of regularity results in the case of coefficients blowing up at $t=0$.
	
	\begin{table}[h!]
		\small\centering
		\begin{tabular}{|ccc|c|cc|c|c|}
			\hline
			\multicolumn{3}{|c|}{\begin{tabular}[c]{@{}c@{}}Rate of\\ Blow-up\end{tabular}}                                         & \multirow{2}{*}{\begin{tabular}[c]{@{}c@{}}Regularity in $t$\\ of coefficients\end{tabular}} & \multicolumn{2}{c|}{\begin{tabular}[c]{@{}c@{}}Growth in $x$\\ of coefficients\end{tabular}} & \multirow{2}{*}{\begin{tabular}[c]{@{}c@{}}Loss of regularity\\ index for solution\end{tabular}} & \multirow{2}{*}{Ref.} \\ \cline{1-3} \cline{5-6}
			
			\multicolumn{1}{|c|}{$p$}                          & \multicolumn{1}{c|}{$q$}                          & $\gamma$ &                                                                                              & \multicolumn{1}{c|}{$\omega$}                            & $\Phi$                            &                                                                                                  &                       \\ \hline
			
			\multicolumn{1}{|c|}{\bf 0}                            & \multicolumn{1}{c|}{\bf 1}                            &  \boldmath$(0,1)$        &  \boldmath$C^1((0,T])$                                                                                 & \multicolumn{1}{c|}{ \boldmath$\omega(x)$}                         &  \boldmath$\Phi(x)$                         & \begin{tabular}[c]{@{}c@{}} \bf Arbitrarily\\ \bf small\end{tabular}                                      & \bf (**)                   \\ \hline
			
			\multicolumn{1}{|c|}{0}                            & \multicolumn{1}{c|}{1}                            & 1              & $C^1((0,T])$                                                                                 & \multicolumn{1}{c|}{1}                                   & 1                                 & Finite                                                                                           & \cite{Cico1}                  \\ \hline
			
			\multicolumn{1}{|c|}{0}                            & \multicolumn{1}{c|}{1}                            & 1              &  $C^1((0,T])$                                                                                 & \multicolumn{1}{c|}{ $\omega(x)$}                         &  $\Phi(x)$                         &  Finite                                                                                           & \cite{RU2}                    \\ \hline
			
			\multicolumn{1}{|c|}{0}                            & \multicolumn{1}{c|}{1}                            &  $(1,\infty)$   &  $C^2((0,T])$                                                                                 & \multicolumn{1}{c|}{ $\omega(x)$}                         &  $\Phi(x)$                         &  Infinite                                                                                         & \cite{RU3}                   \\ \hline
			
			\multicolumn{1}{|c|}{$[0,1)$}                 & \multicolumn{1}{c|}{$(1,\infty)$}                 & -              & $C^1((0,T])$                                                                                 & \multicolumn{1}{c|}{-}                                   & -                                 & Infinite                                                                                         & \cite{CSK}                  \\ \hline
			
			\multicolumn{1}{|c|}{ $\left(0,\frac{1}{2}\right)$} & \multicolumn{1}{c|}{ $\left(1,\frac{3}{2}\right)$} & -              &  $C^1((0,T])$                                                                                 & \multicolumn{1}{c|}{ $\omega(x)$}                         & $\Phi(x)$                         &  Infinite                                                                                         & \cite{RU5}                   \\ \hline
			
		\end{tabular}
		\label{T2}
		\caption{Loss of regularity in case of coefficients blowing-up at $t=0$. (**) refers to the result in this paper.}
	\end{table}

	One can observe from the tables that as $q$ and $\g$ increase so is the loss. In the global setting ($x\in \R^n$ and the coefficients are allowed to grow
	polynomially in $x$), solutions experience loss of both derivatives and decay in a  symmetric fashion. When $q \neq 1,$ we need not consider the logarithmic singularity i.e., the value of $\g$ does not play a major role in deciding the amount of loss. This is due to the fact that $|\ln t|^r \leq t^{-\varepsilon},$ for every $\varepsilon>0$ and $r \in \R.$ The results in first row of Table 2 correspond to the main result of the paper, Theorem \ref{result1}. We show that when the coefficients have a mild blow-up rate as specified by Definition \ref{BUdefn}, one can expect atmost arbitrarily small loss in the regularity index of the solution to (\ref{eq01}) as compared to the Cauchy data.
	
	In this work, we allow the coefficients in (\ref{eq01}) to mildly blow-up near $t=0$ and study the asymptotic behavior of the solutions to the Cauchy problem by considering the metric (discussed in Section \ref{metric})
	\begin{equation}\label{m1}
		g_{\Phi,k} = \P^{-2} \vert dx\vert^2 + \japxik^{-2}\vert d\xi\vert^2,
	\end{equation}
	where  $\japxik=(k^2+|\xi|^2)^{1/2}$ for an appropriately chosen large parameter $k \geq 1.$	
	We report that the solution not only experiences a loss of derivatives but also a decay in relation to the initial datum defined in a Sobolev space (see Section \ref{sb}) tailored to the order of singularity and the metric $g_{\Phi,k}.$
	
	From the energy estimate used in proving the well-posedness we derive an optimal cone condition for the solution of the Cauchy problem (\ref{eq1}) in Section \ref{cone}.  Though the characteristics of the operator $P$ in (\ref{eq1}) are singular in nature, the $L^1$ integrability of sublogarithmic singularity guarantees that the propagation speed is finite. The weight function governing the coefficients also influences the geometry of the slope of the cone in such a manner that the slope grows as $|x|$ grows. Thus, the cone condition in our case is anisotropic in nature.
	
	Our methodology relies upon the conjugation of a first order system corresponding to the operator $P$ in (\ref{eq1}) by a pseudodifferential operator. This operator explains the loss in our context by linking it to the singular behavior and the metric. It is chosen to microlocally compensate the loss of both derivatives and decay. Hence, we call the conjugating operator as \enquote{loss operator}.  In our context, the loss operator is of the form
	\begin{linenomath*}
		\[
		e^{\Theta(t,x,D_x)},
		\]
	\end{linenomath*}
	where $\Theta(t,x,\xi) \in C([0,T];C^\infty(\R^{2n})) \cap C^1((0,T];C^\infty(\R^{2n}))$ is appropriately chosen so that  $\partial_t\Theta \left( \in L^1([0,T];C^\infty(\R^{2n})) \right) $ majorizes the symbols of lower order terms arising after factorization in Section \ref{factr} and satisfies
	\begin{equation}\label{LG}
		\vert \Theta(t,x,\xi)\vert \leq C( \ln(1+\P\japxik))^{\gamma}.
	\end{equation}
	This implies that $	e^{\Theta(t,x,D_x)}$ is a pseudodifferential operator of arbitrary small positive order in both $x$ and $D_x.$ The sublogarithmic order bound in (\ref{LG}) is due to the sublogarithmic singularity of the coefficients as observed in (\ref{logblow}) which is in turn coming from the $O(t^{-1} |\ln t|^{\g-1})$ singularity of their first $t$-derivatives. As the loss of regularity is quantified using the operator $	e^{\Theta(t,x,D_x)},$ the loss is arbitrarily small. We emphasize here that the loss operator gives a scale to measure even the arbitrarily small loss. 
	
	The paper is organized as follows. In Section \ref{tools} we describe the properties of $\Phi,\omega$ and our localization technique.	
	In Section \ref{stmt} we define a Cauchy problem of our interest and state the well-posedness result whose proof will be presented in Section \ref{Proof1}. 
	In Section \ref{Symbol classes}, we define appropriate generalized parameter dependent symbol classes. 
	In Section \ref{cone} we derive a cone condition.

	\section{Tools}\label{tools}
	
	In this section we discuss the main tools of the paper- metrics on the phase space, structure functions $\Phi,\omega$ and a localization technique on the phase space.
	
	\subsection{Our Choice of Metric on the Phase Space}\label{metric}
	
	Let us start by reviewing some notation and terminology used in the study of metrics on the phase space, see \cite[Chapter 2]{Lern} and \cite{nicRodi} for details. Let us denote by $\sigma(X,Y)$ the standard symplectic form on $T^*\R^n\cong \R^{2n}$: if $X=(x,\xi)$ and $Y=(y,\eta)$, then $\sigma$ is given by
	\[
	\sigma(X,Y)=\xi \cdot y - \eta \cdot x.
	\]	
	We can identify $\sigma$ with the isomorphism of $\R^{2n}$ to $\R^{2n}$ such that $\sigma^*=-\sigma$, with the formula $\sigma(X,Y)= \langle \sigma X,Y\rangle$. Consider a Riemannian metric $g_X$ on $\R^{2n}$ (which is a measurable function of $X$) to which we associate the dual metric $g_X^\sigma$ by
		\[ 
		g_X^\sigma(Y)= \sup_{0 \neq Y' \in \R^{2n}} \frac{\langle \sigma Y,Y'\rangle^2}{g_X(Y')}, \quad \text{ for all } Y \in \R^{2n}.
		\]
	
	Considering $g_X$ as a matrix associated to positive definite quadratic form on $\R^{2n}$, $g_X^\sigma=\sigma^*g_X^{-1}\sigma$.
	One defines the Planck function \cite{nicRodi} which plays a crucial role in the development of pseudodifferential calculus as
		\[ 
		h_g(x,\xi) := \sup_{0\neq Y \in \R^{2n}} \Bigg(\frac{g_X(Y)}{g_X^\sigma(Y)}\Bigg)^{1/2}.
		\]
	The uncertainty principle is quantified as the upper bound $h_g(x,\xi)\leq 1$. In the following, we make use of the strong uncertainty principle, that is, for some $\kappa>0$, we have
		\[
		h_g(x,\xi) \leq (1+|x|+|\xi|)^{-\kappa}, \quad (x,\xi)\in \R^{2n}.
		\]
	
	In general, we use the metrics of the form
	\begin{equation}\label{m2}
		g_{\Phi,k}^{\rho,r} = \left( \frac{\japxik^{\rho_2}}{\P^{r_1}} \right)^2 |dx|^2 +  \left( \frac{\P^{r_2}}{\japxik^{\rho_1}} \right)^2 |d\xi|^2.
	\end{equation}
	Here $\rho=(\rho_1,\rho_2)$ , $r=(r_1,r_2)$ for $\rho_j,r_j \in [0,1]$, $j=1,2$ are such that $0 \leq \rho_2<\rho_1 \leq 1$ and $0 \leq r_2<r_1 \leq 1$. 
	The Planck function associated to this metric is  $\P^{r_2-r_1} \japxik^{\rho_2-\rho_1}.$
	
	\subsection{Properties of the Structure Functions $\omega$ and $\Phi$}
	The functions $\o$ and $\P$ are associated with weight and metric respectively. They specify the structure of the differential equation. As pseudodifferential calculus is the datum of the metric satisfying some local and global conditions, in our case it amounts to certain conditions on $\Phi$. The symplectic structure and the uncertainty principle also play a natural role in the constraints imposed on $\Phi$. So we consider $\Phi$ to be a monotone increasing function of $|x|$ satisfying the following conditions: 
		\begin{alignat*}{3}
			1 \; \leq & \quad \Phi(x) &&\lesssim  1+|x| && \quad \text{(sub-linear)} \\
			\vert x-y \vert \; \leq & \quad r\Phi(y) && \implies C^{-1}\Phi(y)\leq \Phi(x) \leq C \Phi(y)  && \quad \text{(slowly varying)} \\
			&\Phi(x+y) && \lesssim  \Phi(x)(1+|y|)^s && \quad \text{(temperate)}
		\end{alignat*}
	for all $x,y\in\R^n$ and for some $r,s,C>0$.
	
	For the sake of calculations arising in the development of symbol calculus related to the metrics $g_{\Phi,k}$, we need to impose following additional conditions:
	\begin{alignat*}{3}
		|\Phi(x) - \Phi(y)| \leq & \Phi(x+y) && \leq \Phi(x) + \Phi(y),  && \quad  (\text{Subadditive})\\
		& |\partial_x^\beta \Phi(x)| && \lesssim \Phi(x) \japx^{-|\beta|}, \\
		&\Phi(ax) &&\leq a\P, \text{ if } a>1,\\
		& a\P &&\leq \Phi(ax), \text{ if } a \in [0,1],
	\end{alignat*}
	where $\beta \in \mathbb{Z}_+^n$. It can be observed that the above conditions are quite natural in the context of symbol classes. In our work, we need even the weight function $\omega$ to satisfy the above stated properties of $\Phi$. In arriving at the energy estimate using the Sharp G\r{a}rding inequality (see Section \ref{energy} for details), we need 
		\[
		\o \lesssim \P, \quad x \in \R^n.
		\]

	\subsection{Sobolev Spaces}\label{sb}
	Let $s=(s_1,s_2) \in \R^2$ and $ k \geq 1.$ We consider the following Sobolev spaces for our work:
	\begin{defn} \label{Sobol}   
		The Sobolev space $\mathcal{H}^{s,\mu,\gamma}_{\Phi,k}(\R^n)$ for $\mu \in \R$ and $\gamma \in (0,1)$ is defined as
		\begin{equation}
			\label{Sobo3}
			\mathcal{H}^{s,\mu,\gamma}_{\Phi,k}(\R^n) = \{v \in L^2(\R^n): \P^{s_2}\langle D \ran^{s_1} e^{\mu(\ln(1+\P\la D_x\rak))^\gamma}v \in L^2(\R^{n}) \},
		\end{equation}
		equipped with the norm
		$
		\Vert v \Vert_{\Phi,k;s,\mu,\gamma} = \Vert \Phi(\cdot)^{s_2}\langle D \ran^{s_1}e^{\mu(\ln(1+\Phi(\cdot)\la D\rak))^\gamma}v \Vert_{L^2} .
		$ 
	\end{defn}
	The subscript $k$ in the notation $\mathcal{H}^{s,\mu,\gamma}_{\Phi,k}(\R^n)$ is related to the parameter in the operator $\la D \ra_k = (k^2 - \Delta_x)^{1/2}.$ Observe that $e^{\mu(\ln(1+\P\la D_x\rak))^\gamma}$ is a pseudodifferential operator of an arbitrarily small positive order in both $x$ and $D_x.$ When $\mu=0,$ the above spaces correspond to the following spaces.
	\begin{defn} \label{Sobo}   
		The Sobolev space $H^{s}_{\Phi,k}(\R^n)$ for $s=(s_1,s_2) \in \R^2$ and $ k \geq 1,$ is defined as
		\begin{equation}
			\label{Sobo2}
			H^{s}_{\Phi,k}(\R^n) = \{v \in L^2(\R^n): \P^{s_2}\langle D \ran^{s_1}v \in L^2(\R^{n}) \},
		\end{equation}
		equipped with the norm
		$
		\Vert v \Vert_{\Phi,k;s} = \Vert \Phi(\cdot)^{s_2}\langle D \ran^{s_1}v \Vert_{L^2} .
		$ 
	\end{defn}
	When $\P$ is bounded and $k=1,$ $H^{s}_{\Phi,1}(\R^n)$ correspond to the usual Sobolev spaces $H^s(\R^n).$
	\begin{rmk}
		For any $\mu \in \R$ and $\gamma \in (0,1),$ we have
		\[
		H^{s+\nu e}_{\Phi,k}(\R^n) \hookrightarrow \mathcal{H}^{s,\mu,\gamma}_{\Phi,k}(\R^n)  \hookrightarrow H^{s-\nu e}_{\Phi,k}(\R^n)
		\]
		where $\nu>0$ is arbitrarily small and $e=(1,1).$
	\end{rmk}

	\subsection{Subdivision of the Phase Space}\label{zones}
	We divide the extended phase space $J=\J,$ where $T>0,$ into two regions using the Planck function $h(x,\xi)=(\P \japxik)^{-1}$ of the metric $g_{\Phi,k}$ in (\ref{m1}). To this end we define $t_{x,\xi}$, for a fixed $(x,\xi)$, as the solution to the equation
	\begin{linenomath*}
		\[
		t=\frac{N}{\P\japxik},
		\]
	\end{linenomath*}
	where $N$ is the positive constant chosen appropriately later. Using $t_{x,\xi}$ we define the interior region
	\begin{linenomath*}
		\begin{equation} \label{zone1}
			\pd =\{(t,x,\xi)\in J : 0 \leq t \leq t_{x,\xi}\}
		\end{equation}
	\end{linenomath*}
	and the exterior region
	\begin{linenomath*}
		\begin{equation} \label{zone2}
			\hyp =\{(t,x,\xi)\in J : t_{x,\xi} < t \leq T\}.
		\end{equation}
	\end{linenomath*}	
	In Section \ref{Symbol classes}, we use these regions to define the parameter dependent global symbol classes.

	\section{Statement of the Main Result}\label{stmt}
	We consider the Cauchy problem:
	\begin{linenomath*}
		\begin{equation}
			\begin{cases}
				\label{eq1}
				P(t,x,\partial_t,D_x)u(t,x)= f(t,x), \qquad D_x = -i\nabla_x,\;(t,x) \in (0,T] \times \R^n, \\
				u(0,x)=f_1(x), \quad \partial_tu(0,x)=f_2(x),
			\end{cases}
		\end{equation}
	\end{linenomath*}
	with the strictly hyperbolic operator $P(t,x,\partial_{t},D_{x}) = \partial_t^2 + a(t,x,D_x)+ b(t,x,D_x)$ where
	\begin{linenomath*}			
		\[
		a(t,x,\xi)  = \sum_{i,j=1}^{n} a_{i,j}(t,x)\xi_i\xi_j \quad \text{ and } \quad
		b(t,x,\xi)  = i\sum_{j=1}^{n} b_{j}(t,x)\xi_j + b_{n+1}(t,x).
		\]
	\end{linenomath*}
	Here, the matrix $(a_{i,j}(t,x))$ is real symmetric for all $(t,x)\in (0,T] \times \R^n$, $a_{i,j} \in C^1((0,T];C^\infty(\R^n))$ and $b_j \in C([0,T];C^\infty(\R^n))$. Similar to the estimates in Remark \ref{rmk}, we have the following assumptions on $a(t,x,\xi)$ and $b(t,x,\xi)$
	\begin{linenomath*}
		\begin{equation}
			\label{conds}
			\begin{rcases}
				\begin{aligned}
					a(t,x,\xi) &\geq C_0 \o^2 \japxik^2, \quad C_0>0, \\
					\vert \partial_\xi^\alpha \partial_x^\beta a(t,x,\xi) \vert &\leq C_{\alpha\beta} \frac{1}{t^{\delta_1}} \o^2 \P^{-\vert \beta \vert}\japxik^{2-\vert \alpha \vert},\quad |\alpha| \geq 0, |\beta|>0, \\
					\vert \partial_\xi^\alpha \partial_x^\beta b(t,x,\xi) \vert &\leq C_{\alpha\beta} \o \P^{-\vert \beta \vert}\japxik^{1-\vert \alpha \vert},
				\end{aligned}
			\end{rcases}
		\end{equation}
	\end{linenomath*}	
\begin{linenomath*}
	\begin{align}
		\label{bl1}
		\vert \partial_\xi^\alpha \partial_x^\beta \partial_t a(t,x,\xi) \vert & \leq C_{\alpha\beta} \frac{(\ln(1+1/t))^{\gamma-1}}{t^{1+\delta_2|\beta|}} \o^2 \P^{-\vert \beta \vert} \japxik^{2-\vert \alpha \vert},\\
		\label{bl2}
		\vert \partial_\xi^\alpha \partial_x^\beta \partial_t a(t,x,\xi) \vert  & \leq C_{\alpha\beta} \frac{1}{t^{\delta_3+\delta_2|\beta|}} \o^2 \P^{-\vert \beta \vert} \japxik^{2-\vert \alpha \vert}\tag{\ref{bl1}${}^*$},
	\end{align}	
	\end{linenomath*}
	where $\delta_1,\delta_2,\delta_3 \in \big[0,1\big), (t,x,\xi) \in [0,T] \times \R^n\times \R^n$. Note that $C_{\alpha\beta}$ is a generic positive constant which may differ at each occurence.
	
	We now state the main result of this paper.
	Let $e=(1,1).$ 	
	\begin{thm}(Arbitrary small loss/ no loss)\label{result1}
		Consider the strictly hyperbolic Cauchy problem  (\ref{eq1}) satisfying the conditions (\ref{conds}) and (\ref{bl1} or \ref{bl2}). Let the initial data $f_j$ belong to $H^{s+(2-j)e}_{\Phi,k}$, $j=1,2$ and the right hand side $f \in C([0,T];H^{s}_{\Phi,k})$.
		Then, denoting $\delta=\max\{\delta_1,\delta_2\}$, for every $\varepsilon \in (0,1-\delta)$ there are $\kappa_0,\kappa_1 >0$ such that for every $s \in \R^2$ there exists a unique global solution
		\begin{linenomath*}
			\[
			u \in C\left([0,T];\mathcal{H}^{s+e,-\Lambda(t),\gamma}_{\Phi,k}\right)\bigcap C^{1}\left([0,T];\mathcal{H}^{s,-\Lambda(t),\gamma}_{\Phi,k}\right),
			\]
		\end{linenomath*}
		where 
		\[
		\Lambda(t) = 
		\begin{cases}
			\begin{aligned}
				&\kappa_0 + \kappa_1 t^\varepsilon/\varepsilon,  \qquad \text{ when (\ref{bl1}) is satisfied,} \\
				&0,   \qquad \qquad \qquad \text{ when (\ref{bl2}) is satisfied.}
			\end{aligned}
		\end{cases}
		\]
		More specifically, the solution satisfies an a priori estimate		
		\begin{linenomath*}
			\begin{equation}
				\begin{aligned}
					\label{est2}
					\sum_{j=0}^{1} \Vert \partial_t^ju(t,\cdot) &\Vert_{\Phi,k;s+(1-j)e,-\Lambda(t),\gamma} \\
					&\leq C \Bigg(\sum_{j=1}^{2} \Vert f_j\Vert_{\Phi,k;s+(2-j)e} 
					+ \int_{0}^{t}\Vert f(\tau,\cdot)\Vert_{\Phi,k;s,-\Lambda(\tau),\gamma}\;d\tau\Bigg)
				\end{aligned}
			\end{equation}
		\end{linenomath*}
		for $0 \leq t \leq T, \; C=C_s>0$.				
	\end{thm}

	\section{Parameter Dependent Global Symbol Classes}	\label{Symbol classes}
	We now define certain parameter dependent global symbols that are associated with the study of the Cauchy problem (\ref{eq1}). Let $m=(m_1,m_2)\in \mathbb{R}^2$. Consider the metric $g_{\Phi,k}^{\rho,r}$ as in (\ref{m2}).
	
	\begin{defn}
		$G^{m_1,m_2}(\omega,g_{\Phi,k}^{\rho,r})$ is the space of all functions $p=p(x,\xi) \in C^\infty(\mathbb{R}^{2n})$ satisfying 
		\begin{linenomath*}
			\begin{equation*}
				\label{sym1}
				\sup_{\alpha,\beta \in \mathbb{N}^n} \sup_{(x,\xi)\in \R^n}  \japxik^{-m_1+\rho_1|\alpha|-\rho_2|\beta|} \o^{-m_2} \P^{r_1|\beta|-r_2|\alpha|} |\partial_\xi^\alpha  D_x^\beta p(x,\xi)| < +\infty
			\end{equation*}	
		\end{linenomath*}		
	\end{defn}
	We denote the metric $g^{(1,0),(1,0)}_{\Phi,k}$ and the corresponding symbol class $G^{m_1,m_2}(\omega,g^{(1,0),(1,0)}_{\Phi,k})$ as $g_{\Phi,k}$ and $G^{m_1,m_2}(\omega,g_{\Phi,k})$, respectively.
	
	To handle the stronger singular behavior of the characteristics of the operator $P$ in (\ref{eq1}), we have the following symbol classes. Denote
	\begin{equation}\label{theta}
		\theta(t) = \left( \ln \left( 1+\frac{1}{t}\right)\right)^\gamma.
	\end{equation}
	\begin{defn}
		$G^{m_1,m_2}\{l_1,l_2;\g,p\}_{int,N}(\omega,g_{\Phi,k})$ for $l_1,l_2\in \R$ and $p\in[0,1)$is the space of all $t$-dependent symbols $a=a(t,x,\xi)$ in $C^1((0,T];G^{m_1,m_2}(\omega,g_{\Phi,k}))$ satisfying 
		\begin{linenomath*}
			\begin{align*}
				\vert  \partial_\xi^\alpha a(t,x,\xi) \vert & \leq C_{00} \japxik^{m_1-|\alpha|} \o^{m_2}\theta(t)^{l_1},\\
				\vert \partial_\xi^\alpha D_x^\beta a(t,x,\xi) \vert  &\leq C_{\alpha \beta} \japxik^{m_1-|\alpha|} \o^{m_2}\P^{-|\beta|} \left(\frac{1}{t}\right) ^{pl_2} ,
			\end{align*}
		\end{linenomath*}
		for all $(t,x,\xi)\in \pd $ and for some $C_{\alpha \beta}>0$ where $\alpha\in \mathbb{N}^n_0$ and $\beta \in \mathbb{N}^n.$ 
	\end{defn}	
	
	\begin{defn}
		$G^{m_1,m_2}\{l_1,l_2,l_3,l_4;\g,p\}_{ext,N}(\omega,g_{\Phi,k})$ for $l_j\in \R,j=1,\dots,4$ and $p\in[0,1)$ is the space of all $t$-dependent symbols $a=a(t,x,\xi)$ in $C^1((0,T]; G^{m_1,m_2}(\omega,g_{\Phi,k}))$ satisfying 
		\begin{linenomath*}
			\begin{align*}
				\vert &\partial_\xi^\alpha D_x^\beta a(t,x,\xi) \vert \\
				& \leq C_{\alpha \beta} \japxik^{m_1-|\alpha|} \o^{m_2}\P^{-|\beta|} \bigg(\frac{1}{t}\bigg)^{l_1+p(l_2+|\beta|)}   \theta(t)^{l_3+l_4(|\alpha|+|\beta|)} 
			\end{align*}
		\end{linenomath*}
		for all $(t,x,\xi)\in \hyp$ and for some $C_{\alpha \beta}>0$ where $\alpha,\beta \in \mathbb{N}^n_0.$ 
	\end{defn}	
	Given a $t$-dependent global symbol $a(t,x,\xi)$, we can associate a pseudodifferential operator $Op(a)=a(t,x,D_x)$ to $a(t,x,\xi)$ by the following oscillatory integral
	\begin{linenomath*}
		\begin{align*}
			a(t,x,D_x)u(t,x)& =\iint\limits_{\mathbb{R}^{2n}}e^{i(x-y)\cdot\xi}a(t,x,\xi){u}(t,y)dy \textit{\dj}\xi \\
			& = (2\pi)^{-n}\int\limits_{\mathbb{R}^n}e^{ix\cdot\xi}a(t,x,\xi)\hat{u}(t,\xi) \textit{\dj}\xi,
		\end{align*}
	\end{linenomath*}
	where $\textit{\dj}\xi = (2 \pi)^{-n}d\xi$ and $\hat u$ is the Fourier transform of $u$ in the space variable.
	
	We denote the class of operators with symbols in  $G^{m_1,m_2}(\omega,g_{\Phi,k}^{\rho,r})$ by  $OPG^{m_1,m_2}(\omega,g_{\Phi,k}^{\rho,r})$. We refer to \cite[Section 1.2 \& 3.1]{nicRodi} and \cite{CT} for the calculus of such operators. The calculus for the operators with symbols of form $a(t,x,\xi) = a_1(t,x,\xi) + a_2(t,x,\xi)$ such that
	\begin{linenomath*}
		\[
		\begin{aligned}
			a_1 &\in G^{\tilde m_1,\tilde m_2}\{\tilde l_1,\tilde l_2;\g,\delta_1\}_{int,N_1} \wm, \\
			a_2 &\in G^{m_1,m_2}\{l_1,l_2,l_3,l_4;\g,\delta_2\}_{ext,N_2}\wm,
		\end{aligned}
		\]
	\end{linenomath*}
	for $N_1 \geq N_2,$ follows in similar lines to the one in \cite[Appendix]{RU2}.

	\section{Proofs} \label{Proof1}
	We first give the proof of Theorem \ref{result1} when (\ref{bl1}) is satisfied and the case for (\ref{bl2}) follows in similar lines.  The proof in our context follows in similar lines to the one given in \cite[Section 5]{RU2}.
	
	There are three key steps in the proof of Theorem \ref{result1}. First, we factorize the operator $P(t,x,\partial_t,D_x)$. To this end, we begin with modifying the coefficients of the principal part by performing an excision so that the resulting coefficients are regular at $t=0$. Second, we reduce the original Cauchy problem to a Cauchy problem for a first order system (with respect to $\partial_t$). Lastly, using sharp G\r{a}rding’s inequality we arrive at the $L^2$ well-posedness of a related auxiliary Cauchy problem, which gives well-posedness of the original problem in the Sobolev spaces $\mathcal{H}^{s,\mu,\gamma}_{\Phi,k}(\R^n)$.
	
	\subsection{Factorization}\label{factr}
	From the estimate (\ref{bl1}), we observe that $a(t,x,\xi)$ is sublogarithmically bounded near $t=0$, i.e.,
	\begin{linenomath*}
		\begin{equation}
			\label{log}
			|a(t,x,\xi)| \leq C  \o^{2} \japxik^2 \theta(t), \quad C>0.
		\end{equation}
	\end{linenomath*}
	We modify the symbol $a$ in $Z_{int}(2)$, by defining
	\begin{linenomath*}
		\begin{equation}\label{exci}
			\tilde{a}(t,x,\xi)=\varphi(t\P \japxik)\o^{2} \japxik^{2} + (1-\varphi(t\Phi(x) \japxik))a(t,x,\xi)
		\end{equation}
	\end{linenomath*} 
	for $
	\varphi \in C^\infty(\mathbb{R}) \text{ , }0\leq \varphi \leq 1 \text{ , } \varphi=1 \text{ in }[0,1] \text{ , }\varphi=0 \text{ in }[2,+\infty).$
	Note that 
	$$(a-\tilde{a}) \in G^{2,2}\{1,1;\g,\delta_1\}_{int,2}(\omega,g_{\Phi,k}) \text{ and } (a-\tilde{a}) \sim 0 \text{ in } Z_{ext}(2).$$ This
	implies that $t^{\delta_1}(a-\tilde{a})$ for $ t \in [0,T]$ is a bounded and continuous family in $G^{2,2}(\omega,g_{\Phi,k})$. 
	Observe that $a-\tilde{a}$ is $L^1$ integrable in $t$, i.e.,
	\begin{equation}\label{diff}
		\begin{aligned}
			\int^{T}_{0}\vert (a-\tilde{a})(t,x,\xi)\vert dt 
			&\leq \kappa_{0}  \o^{2} \japxik^{2} \left|\int_{0}^{2/\Phi(x) 	\japxik}\theta(t)dt\right|	\\
			&\leq \kappa_{0} \o \japxik (\ln(1+\Phi(x) \japxik))^\gamma.
		\end{aligned}
	\end{equation}
	
	Let $\tau(t,x,\xi)= \sqrt{\tilde{a}(t,x,\xi)}$ and $\delta=\max\{\delta_1,\delta_2\}$. It is easy to note that 
	\begin{enumerate}[label=\roman*)]
		\item $\tau(t,x,\xi)$ is $G_\omega$-elliptic symbol of order $(1,1)$ i.e. there is $C>0$ such that for all $(t,x,\xi)\in [0,T] \times \mathbb{R}^n \times \mathbb{R}^n$ we have
		\begin{linenomath*}
			\[
			\vert\tau(t,x,\xi)\vert \geq C  \o \japxik.
			\]
		\end{linenomath*}
		\item $\tau(t,x,\xi) \in G^{1,1}\{0,0;\g,0\}_{int,2}(\omega,g_{\Phi,k}) + G^{1,1}\{0,0,1,1;\g,\delta_1\}_{ext,1}(\omega,g_{\Phi,k}) $.
		
		\item By definition
		\[
		\partial_t \tau(t,x,\xi) = \frac{1}{2\tau} \left[  \P\japxik\varphi'(t\P\japxik) (\o^2\japxik^2-a) + (1-\varphi(t\P\japxik)) \partial_ta \right].
		\]
		Hence,
		\begin{alignat}{3}
				\label{tau1}
				&\partial_t \tau &&\sim 0 && \quad  \text{ in } Z_{int}(1), \\
				\label{tau2}
				&\partial_t \tau && \in G^{2,2}\{1,1;\gamma,\delta_1\}_{int,2}(\Phi,g_{\Phi,k})&& \quad  \text{ in } Z_{int}(2)\setminus Z_{int}(1),\\
				\label{tau3}
				&\partial_t \tau &&\in G^{1,1}\{1,0,1,1;\g,\delta\}_{ext,2}(\omega,g_{\Phi,k}) &&  \quad \text{ in } Z_{ext}(2).
		\end{alignat}
		To be precise,  there are $C_0,C_{\alpha \beta}>0$ such that for $(t,x,\xi)\in [0,T] \times \mathbb{R}^n \times \mathbb{R}^n$ and $|\alpha|\geq0,|\beta|>0$ we have
		\begin{linenomath*}
			\begin{equation}
				\begin{rcases}
				\label{tau4}
				\begin{aligned}
					|\partial_\xi^\alpha D_x^\beta \partial_t \tau(t,x,\xi)| &\sim \chi_{int}(1) \; 0 \\
					| \partial_t \tau(t,x,\xi)| &\leq C_0\left( \chi_{int}(2)- \chi_{int}(1) \right) \P^2 \japxik^2 \; \t \\
					|\partial_\xi^\alpha D_x^\beta \partial_t \tau(t,x,\xi)| & \leq C_{\alpha \beta} \left( \chi_{int}(2)- \chi_{int}(1) \right) \P^2 \japxik^2 \; \frac{1}{t^{\delta_1}} \\
					{|\partial_t\tau(t,x,\xi)|}  & \leq C_{0} \; \chi_{ext}(1)\japxik \o  \frac{\t}{t},\\
					{|\partial_\xi^\alpha D_x^\beta \partial_t \tau(t,x,\xi)|} & \leq C_{\alpha \beta} \; \chi_{ext}(1)\japxik^{1-|\alpha|} \o \P^{-|\beta|}  \frac{\t^{1-\frac{1}{\g}}}{t} \frac{\t^{|\alpha|+|\beta|}}{t^{\delta|\beta|}}.
				\end{aligned}
			\end{rcases}
			\end{equation}
		\end{linenomath*}
	\end{enumerate}  
	Here $\chi_{int}(N_1)$ and  $\chi_{ext}(N_2)$ are the indicator functions for the regions $Z_{int}(N_1)$ and $Z_{ext}(N_2)$, respectively.
	From the properties (i-iii) of $\tau$ and by the definition of $\tilde a$ in (\ref{exci}), we have the following two lemmas.
	\begin{lem}
		Let $\varepsilon,\varepsilon'$ be such that $0<\varepsilon< \varepsilon'<1-\delta$. Then,
		\begin{enumerate}[label=\roman*)]
			\item $\tau \in C([0,T]; G^{1+\varepsilon,1}(\omega\Phi^\varepsilon,g^{(1,\delta_1),(1-\delta_1,0)}_{\Phi,k})$, 
			\item $\tau^{-1} \in C([0,T]; G^{-1,-1}(\omega,g^{(1,\delta_1),(1-\delta_1,0)}_{\Phi,k})$,
			\item $t^{1-\varepsilon}\partial_t \tau(t,\cdot,\cdot) \in G^{1+\varepsilon',1+\varepsilon'}(\Phi,g_{\Phi,k}^{(1,\delta),(1-\delta,0)})$, for all $t \in [0,T]$.
		\end{enumerate}	
	\end{lem}
	\begin{proof}
		The first two claims follow from \cite[Prposition A.1]{RU2} while the third from the observation that $t^{1-\varepsilon} \left(  \frac{1}{t} \t^{1+|\alpha|+|\beta|}  \right) \leq \frac{1}{t^{\varepsilon'}} \leq \big(\P\japxik \big)^{\varepsilon'}$ in $Z_{ext}(1).$
	\end{proof}
	
	\begin{lem}
		Let $\varepsilon$ be such that $0<\varepsilon<1-\delta$. Then,
		\begin{enumerate}[label=\roman*)]
			\item $t^{1-\varepsilon}(\tilde{a}(t,x,D_x) -\tau(t,x,D_x) ^2) \in C\Big([0,T];OPG^{1,1}(\omega,g_{\Phi,k}^{(1,\delta_1),(1-\delta_1,0)})\Big) $,
			\item $t^{1-\varepsilon} (a(t,x,D_x)-\tilde a(t,x,D_x)) \in C\Big([0,T];OPG^{1,1}(\omega,g_{\phi,k})\Big).$
		\end{enumerate}
	\end{lem}
	\begin{proof}
		The proof is a consequence of the fact that $t^{1-\varepsilon} \t^{1+|\alpha|+|\beta|}$ is bounded and continuous for all $t \in [0,T]$ and for all $\alpha,\beta \in \mathbb{N}_0^n.$ 
	\end{proof}
	
	We have the following factorization of the operator $P(t,x,\partial_t,D_x)$
	\begin{linenomath*}
		\begin{equation*}
			P(t,x,\partial_t,D_x) = (\partial_t-i\tau(t,x,D_x)) 	(\partial_t+i\tau(t,x,D_x))+ (a-\tilde{a})(t,x,D_x) + a_1(t,x,D_x)
		\end{equation*}
	\end{linenomath*}
	where the operator $a_1(t,x,D_x)$ is such that, for $t \in [0,T]$, 
	\begin{linenomath*}
		\begin{equation*}
			a_1 = -i[\partial_t,\tau] + \tilde{a} -\tau^2 + b \text{ and } t^{1-\varepsilon}a_1(t,x,D_x) \in OPG^{1+\varepsilon',1+\varepsilon'}(\Phi,g_{\Phi,k}^{(1,\delta),(1-\delta,0)}).
		\end{equation*}
	\end{linenomath*}

	\subsection{First Order Pseudodifferential System}
	Next we obtain a first order $2\times2$ pseudodifferential system equivalent to the operator $P$ by following a procedure similar to the one used in \cite{RU2,Cico1}. To this end, we introduce the change of variables $U=U(t,x)=(u_1(t,x),u_{2}(t,x))^T$, where
	\begin{equation}
		\label{COV}
		\begin{cases}
			u_1(t,x)= (\partial_t+i\tau(t,x,D_x))u(t,x), \\ 
			u_2(t,x)=  \o \la D_x \rak u(t,x) - H(t,x,D_x)u_1, \\  
		\end{cases}
	\end{equation}
	and the operator $H$ with the symbol $\sigma(H)(t,x,\xi)$ is such that
	\begin{linenomath*}
		\[
		\sigma(H)(t,x,\xi) = -\frac{i}{2}\o \japxik  \frac{\Big(1-\varphi\Big(t\P \japxik/3\Big)\Big)}{\tau(t,x,\xi)}.
		\]
	\end{linenomath*}
	By the definition of $H$,  $\text{supp } \sigma(H) \cap \text{supp } \sigma(a-\tilde{a}) = \emptyset$ and 
	\begin{linenomath*}
		\begin{align*}
			\sigma(2iH(t,x,D_x) \circ \tau(t,x,D_x) )&\sim 0,  \quad \text{ in } Z_{int}(3),\\
			\sigma(2iH(t,x,D_x) \circ \tau(t,x,D_x) )&= \o \la \xi \ra_k (1+\sigma(K_1)), \quad \text{ in } Z_{ext}(3),
		\end{align*}
	\end{linenomath*}
	where $\sigma(K_1) \in G^{-1,-1}\{0,0;\g,\delta_1\}_{int,6}(\omega,g_{\Phi,k}) + G^{-1,-1}\{0,1,2,1;\g,\delta_1\}_{ext,3}(\omega,g_{\Phi,k})$. Then, the equation $Pu=f$ is equivalent to the first order $2\times2$ system  :
	\begin{equation}
		\label{FOS1}
		\begin{aligned}
			LU &= (\partial_t + \mathcal{D}+A_0+A_1)U=F,\\
			U(0,x)&=(f_2+i\tau(0,x,D_x)f_1,\P\la D_x\ra f_1)^T ,
		\end{aligned}
	\end{equation}
	where
	\begin{linenomath*}
		\begin{align*}
			F&=(f(t,x) ,-H(t,x,D_x)f(t,x) )^T,\\
			\mathcal{D} &= \text{diag}(-i\tau(t,x,D_x),i\tau(t,x,D_x)),\\
			A_0 &= \begin{pmatrix}
				B_0H & B_0 \\
				-HB_0H & \quad HB_0
			\end{pmatrix}
			= \begin{pmatrix}
				\mathcal{R}_1 & B_0 \\
				-\mathcal{R}_3 & \mathcal{R}_2
			\end{pmatrix},\\
			A_1 &= \begin{pmatrix}
				B_1H & B_1 \\
				B_2 & \qquad i[M,\tau]M^{-1}-HB_1
			\end{pmatrix}.
		\end{align*}
	\end{linenomath*}
	The operators $M, M^{-1},B_0,B_1$ and $B_2$ are as follows
	\begin{linenomath*}
		\begin{align*}
			M &= \o\la D_x \rak, \quad
			M^{-1} = \la D_x \rak^{-1}  \o^{-1},  \\
			B_0 &= (a(t,x,D_x) -\tilde a (t,x,D_x))\la D_x \rak^{-1}  \o^{-1},  \\
			B_1 &= (-i\partial_t\tau(t,x,D_x) + \tilde{a}(t,x,D_x) -\tau(t,x,D_x)^2 + b(t,x,D_x)) \la D_x \rak^{-1}  \o^{-1},\\
			B_2 &= 2iH\tau-M+i[M ,\tau]M^{-1}H + i[\tau, H]-HB_1H+\partial_tH.
		\end{align*}
	\end{linenomath*}
	By the definition of operator $H$, we have $B_0H = \mathcal{R}_1, HB_0 =\mathcal{R}_2$, $HB_0H=\mathcal{R}_3$ for $\mathcal{R}_j\in G^{-\infty,-\infty}(\omega, g_{\Phi,k}),j=1,2,3,$ and the operator $2iH\tau-M$ is such that
	\begin{linenomath*}
		\[
		\sigma(2iH\tau-M) = \begin{cases}
			-\o\la \xi \rak, & \text{ in } Z_{int}(3),\\
			\o\la \xi \rak \sigma(K_1), & \text{ in } Z_{ext}(3).
		\end{cases}
		\]
	\end{linenomath*}
	The operators $\mathcal{D}, A_0$ and $A_1$ are such that
	\begin{linenomath*}
		\begin{equation}\label{As2}
			\begin{rcases}
				\begin{aligned}
					\sigma(\mathcal{D}) &\in G^{1,1}\{0,0;\g,0\}_{int,2}(\omega,g_{\Phi,k}) + G^{1,1}\{0,0,1,1;\g,\delta_1\}_{ext,1}(\omega,g_{\Phi,k}), \\
					\sigma(A_0) &\in G^{1,1}\{1,1;\g,\delta_1\}_{int,2}(\omega,{g}_{\Phi,k}) + G^{-\infty,-\infty}\{0,0,0,0;\g,0\}_{ext,3}(\omega,{g}_{\Phi,k}), \\
					\sigma(A_1) & \in G^{1,1}\{1,1;\g,\delta_1\}_{int,6}(\Phi,{g}_{\Phi,k}) + G^{0,0}\{1,0,1,1;\g,\delta\}_{ext,1}(\omega,{g}_{\Phi,k})\\ & \qquad+G^{0,0}\{0,1,2,1;\g,\delta\}_{ext,3}(\omega,{g}_{\Phi,k}),
				\end{aligned}
			\end{rcases}
		\end{equation}
	\end{linenomath*}
	where the last expression in (\ref{As2}) is obtained from the symbolic estimates on $\partial_t \tau(t,x,\xi)$ as given in (\ref{tau1}-\ref{tau4})
	and thus, by \cite[Propositions A.1-A.2, Remarks A.2-A.3]{RU2}, for every $\varepsilon < 1-\delta,$
	\begin{linenomath*}
		\begin{equation}\label{As}
			\begin{rcases}
				\begin{aligned}
					t^{1-\varepsilon}\sigma(A_0(t)) &\in C\left([0,T]; G^{1,1}(\omega,{g}_{\Phi,k}) \right),\\
					t^{1-\varepsilon}\sigma(A_1(t)) &\in 	C\left([0,T];G^{\varepsilon',\varepsilon'}(\Phi,{g}_{\Phi,k}^{(1,\delta),(1-\delta,0)}) \right).
				\end{aligned}
			\end{rcases}
		\end{equation}
	\end{linenomath*}
	We define positive functions $\psi_0, \psi _1\in L^1([0,T];C^\infty(\R^n)) \cap C^1((0,T];C^\infty(\R^n))$ 
	\begin{equation}\label{tpsi}
		\begin{aligned}
			\psi_0(t,x,\xi) &=  C_0\varphi(t\P\japxik/3)\t\o\japxik, \\
			\psi_1(t,x,\xi) &=  C_1\left(\varphi(t\P\japxik/3)\t \P\japxik + (1-\varphi(t\P\japxik))\frac{\t^{1-\frac{1}{\gamma}}}{t} \right),
		\end{aligned}
	\end{equation}
	for appropriate $C_0,C_1>0$, satisfying the estimates
	\[
	|\sigma(A_0)|  \leq \psi_0 \quad \text{ and } \quad |\sigma(A_1)|  \leq \psi_1.
	\]
	The function $\psi =  \psi_0+ \psi_1$ satisfies
	\begin{equation}
		\label{bd}
		\begin{aligned}
			\int^{T}_{0}| \psi (t,x,\xi)|dt
			&\leq \kappa_{00} (\ln(1+\Phi(x) \japxik))^{\gamma}\quad \text{ and } \\
			\int^{t}_{0}|\partial_\xi^\alpha D_x^\beta \psi (t,x,\xi)|dt
			&\leq \kappa_{\alpha \beta}\Phi(x)^{-|\beta|} \japxik^{-|\alpha|} (\ln(1+\Phi(x) \japxik))^\gamma \chi_{int}(6),
		\end{aligned}
	\end{equation}
	for $|\alpha| + |\beta| >0.$
	
	\subsection{Energy Estimate} \label{energy}
	It is sufficient to consider the case $s=(0,0)$ as the operator  
	\begin{linenomath*}
	$$\P^{s_2} \la D\rak^{s_1}L\la D\rak^{-s_1}\P^{-s_2}, \quad s=(s_1,s_2)$$ 
	\end{linenomath*}
	satisfies the same hypotheses as $L$.
	In the following, we establish some lower bounds for the operator $\mathcal{D}+A_0+A_1$. The symbol $d(t,x,\xi)$ of the operator $\mathcal{D}(t)+\mathcal{D}^*(t)$ is such that
	\begin{linenomath*}
		\[
		d \in G^{0,0}\{0,0;\g,0\}_{int,2}(\omega,g_{\Phi,k}) + G^{0,0}\{0,0,1,1;\g,\delta_1\}_{ext,1}(\omega,g_{\Phi,k}).
		\]
	\end{linenomath*}
	It follows from \cite[Proposition A.1 and Remark A.4]{RU2} that 
	\begin{linenomath*}
		\[
		t^{1-\varepsilon}d \in C([0,T];G^{0,0}(\omega,g_{\Phi,k}).
		\]
	\end{linenomath*}
	Thus
	\begin{equation}\label{lb1}
		2\Re \la \mathcal{D}{U},{U} \ra_{L^2} \geq -\frac{C}{t^{1-\varepsilon}} \la {U},{U} \ra_{L^2}, \quad C>0.
	\end{equation}
	
	We perform a following change of variable
	\begin{equation}\label{c1}
		V_1(t,x)= e^{-\int_{0}^{t} \psi(r,x,D_x)dr}{U}(t,x),
	\end{equation}
	where $ \psi(t,x,\xi)$ is as in (\ref{tpsi}). Observe that the operator $e^{\pm\int_{0}^{t} \psi(r,x,D_x)dr}$ is a pseudodifferential operator of arbitrarily small positive order . Applying \cite[Lemma A.7]{RU2} to the identity operator we see that
	\begin{equation}\label{LRI}
		\begin{aligned}
			e^{\int_{0}^{t} \psi(r,x,D_x)dr} e^{-\int_{0}^{t} \psi(r,x,D_x)dr} & = I + K_2^{(1)}(t,x,D_x),\\
			e^{-\int_{0}^{t} \psi(r,x,D_x)dr} e^{\int_{0}^{t} \psi(r,x,D_x)dr} & = I + K_2^{(2)}(t,x,D_x),
		\end{aligned}
	\end{equation}
	where for every $\tilde\varepsilon \ll 1$ and $j=1,2,$
	\begin{linenomath*}
		$$
		\begin{aligned}
			(\ln(1+\P\japxik))^{-\gamma} \sigma(K_2^{(j)}) \in &G^{(-1+\tilde \varepsilon)e}\{0,0;\g,0\}_{int,6} (\Phi, g_{\Phi,k}) \\ &\quad + G^{-e}\{0,0,0,0;\g,0\}_{ext,1} (\Phi, g_{\Phi,k}).
		\end{aligned}
		$$  
	\end{linenomath*}
	By \cite[Proposition A.2]{RU2},
	$\sigma(K_2^{(j)})\in G^{(-1+\varepsilon)e}(\omega;g_{\Phi,k}).$ We choose $k>k_1$ for large $k_1$ so that the operator norm of $K_2^{(j)}, j=1,2$ is strictly lesser than $1$ and the existence of 
	\begin{equation}\label{inv1}
		(I+K_2^{(j)}(t,x,D_x))^{-1} = \sum_{l=0}^{\infty} (-1)^jK_2^{(j)}(t,x,D_x)^l, \quad j=1,2, 
	\end{equation}
	is guaranteed. The equation (\ref{c1}) implies that
	\begin{equation}
		\label{IstCV}
		\begin{rcases}
			\begin{aligned}
				U(0,x) & = V_1(0,x),\\
				\Vert{U}(t,\cdot)\Vert_{\Phi,k;(0,0),-\kappa_0,\g} &\leq 2\Vert V_1(t,\cdot)\Vert_{L^2} \text{ , } \kappa_0>0 	\text{ , }0<t\leq T, \\
				U(t,x) &= (I+K_2^{(1)}(t,x,D_x))^{-1} e^{\int_{0}^{t} \psi(r,x,D_x)dr}V_1(t,x).
			\end{aligned}
		\end{rcases}
	\end{equation}
	Here $\kappa_0$ is same as $\kappa_{\alpha\beta}$ appearing in (\ref{bd}) for $\alpha=\beta=0$ . For $L$ as in (\ref{FOS1}), we have 
	\begin{linenomath*}
		\[
		LU = (\partial_t+\mathcal{D}+A)(I+K_2^{(1)}(t,x,D_x))^{-1} e^{\int_{0}^{t} \psi(r,x,D_x)dr}V_1 = F, \; A=A_0+A_1.
		\] 
	\end{linenomath*}
	Use of \cite[Lemma A.7]{RU2} yields a first order pseudodifferential system, formally equivalent to the one in (\ref{FOS1}), $L_1V_1 = F_1$ where
	\begin{linenomath*}
		\begin{equation*} 
			\begin{rcases}
				L_1 = \partial_t+\mathcal{D}+ \psi I+A+R_1,\; A=A_0+A_1,\\
				F_1 = (I+K_2^{(2)}(t,x,D_x))^{-1}e^{-\int_{0}^{t} \psi(r,x,D_x)dr} (I+K_2^{(1)}(t,x,D_x)) F.		
			\end{rcases}
		\end{equation*}
	\end{linenomath*}
	Here the operators $\mathcal{D}, A_0,A_1$ are as in (\ref{FOS1})-(\ref{As2}). Due to \cite[Lemma A.7]{RU2}, for an arbitrary small $\tilde\varepsilon>0,$
	\begin{linenomath*}
		\[
		\begin{aligned}
			(\ln&(1+\P\japxik))^{-\g} \sigma(R_1) \\
			&\in G^{\tilde\varepsilon,1}\{1,1;\g,\delta_1\}_{int,6} (\omega\Phi^{-1+\tilde\varepsilon}, g_{\Phi,k}) + 	G^{0,1}\{0,0,1,1;\g,\delta\}_{ext,1} (\omega\Phi^{-1}, g_{\Phi,k})\\
			& \quad + G^{-1,1}\{1,0,1,1;\g,\delta\}_{ext,1} (\Phi^{-1}, g_{\Phi,k}) + G^{-1,1}\{0,1,2,1;\g,\delta\}_{ext,3} (\Phi^{-1}, g_{\Phi,k}).
		\end{aligned}
		\]	
	\end{linenomath*}
	Using the compensation procedure outlined in \cite[Remark A.4]{RU2}, one can show that 
	\begin{linenomath*}
		$$
		t^{1-\varepsilon}(\ln(1+\Phi(x) \la \xi \ran))^{-\gamma} \sigma(R_1) \in {C([0,T];G^{0,0}(\Phi;g^{(1,\delta),(1-\delta,0)}_{\Phi,k})}, \quad 0<\varepsilon<1-\delta.
		$$
	\end{linenomath*}
	
	For an appropriate choice of $C_0,C_1>0$ in the definition of $ \psi$ as in (\ref{tpsi}-\ref{bd}), we observe that $ \psi I + A $ satisfies 
	\begin{linenomath*}
		\begin{align*}
			2 \psi I+\sigma(A+&A^*)\geq 0,\\
			t^{1-\varepsilon}( \psi I+\sigma(A)) &\in C([0,T];G^{\varepsilon',1}(\Phi^{\varepsilon'}, g_{\Phi,k}^{(1,\delta),(1-\delta,0)})).
		\end{align*}
	\end{linenomath*}
	Here $A=A_0+A_1$ with $A_0$ and $A_1$ as in (\ref{FOS1}-\ref{As}).
	We now apply sharp G\r{a}rding inequality (see \cite[Theorem 18.6.14]{Horm} to the operators $2{\psi}_0I+A_0$ and $2{\psi}_1I+A_1$ separately. The symbols of these operators are governed by the metrics $g_{\Phi,k}$  and $g_{\Phi,k}^{(1,\delta),(1-\delta,0)}$ respectively where the respective Planck functions are $h(x,\xi)=(\P \la \xi \rak)^{-1}$ and $\tilde h(x,\xi)=\Phi(x)^{-1+\delta} \la \xi \rak ^{-1+\delta}.$ Notice that the symbol $\sigma(A_0)$ has the weight function $\o\japxik$ while the Planck function of the governing metric is given by $ h(x,\xi)$. Hence, for the application of sharp G\r{a}rding inequality, we need 
	\begin{linenomath*}
		\[
		\o \lesssim \P.
		\]
	\end{linenomath*}
	Ensuring this yields 
	\begin{equation}
		\label{lb2}
		2\Re \la (\psi I+A)V_1, V_1\ra_{L^2} \geq -Ct^{-1+\varepsilon}\langle V_1,V_1 \rangle_{L^2} \text{ , }C>0.
	\end{equation}

	As for the operator $R_1$, since the symbol $t^{1-\varepsilon}(\ln(1+\Phi(x) \langle \xi \rangle))^{-\g}R_1 $ is uniformly bounded, for a large choice of $\kappa_1$, the application of sharp G\r{a}rding inequality yields		
	\begin{equation}
		\label{R1}
		2 \Re \langle R_1V_1,V_1 \rangle_{L^2} \geq -\frac{\kappa_1}{t^{1-\varepsilon}} \big(2 \Re \la (\ln(1+\Phi(x) \la D_x \rak))^\g V_1,V_1 \ra_{L^2} + \Vert V_1\Vert_{L^2}\big).
	\end{equation}
	We make a further change of variable
	\begin{linenomath*}
		\begin{equation}\label{c2}
			V_2(t,x)= e^{-\mu(t) (\ln(1+\Phi(x) \la D_x \rak))^{\g} }V_1(t,x), \qquad  \mu(t)=\kappa_1t^\varepsilon/\varepsilon,
		\end{equation}
	\end{linenomath*}
	where $\kappa_1$ is the constant as in (\ref{R1}). Let
	\begin{equation*} 
		e^{\pm\mu(t) (\ln(1+\Phi(x) \la D_x \rak))^{\g} } e^{\mp \mu(t) (\ln(1+\Phi(x) \la D_x \rak))^{\g} } = I + K_3^{(\pm)}(t,x,D_x),
	\end{equation*}
	where $\sigma(K_3^{(\pm)}) \in C([0,T]; G^{-1,-1}(\omega, g_{\Phi,k}) )$. As in (\ref{inv1}), we choose $k>k_2$, $k_2$ large, so that $(I+K_3^{(\pm)}(t,x,D_x))^{-1}$ exists. From now on we fix $k$ such that $k>\max\{k_1,k_2\}$. Further, note that
	\begin{linenomath*}
		\begin{equation}\label{V1}
			\begin{rcases}
				\begin{aligned}
					V_2(0,x) &= U(0,x), \\
					\Vert{U}(t,\cdot)\Vert_{\Phi,k;(0,0),-\Lambda(t),\g} &\leq 2^{\mu(T)+1}\Vert V_2(t,\cdot)\Vert_{L^2}, \quad \Lambda(t)=\kappa_0+\kappa_1t^\varepsilon/\varepsilon \text{ , } 0 < t \leq T,\\
					V_1(t,x) &= (1+K_3^{(+)}(t,x,D_x))^{-1} e^{\mu(t) (\ln(1+\Phi(x) \la D_x \rak))^{\g} }V_2(t,x).
				\end{aligned}
			\end{rcases}
		\end{equation}
	\end{linenomath*}
	This implies that $L{U}=F$ if and only if $L_2V_2=F_2$ where
	\begin{equation}\label{last}
		\begin{rcases}
			L_2= \partial_t + \mathcal{D}+( \psi I+A)+(\kappa_1t^{-1+\varepsilon} (\ln (1+\Phi(x) \la 	D_x\ran))^\g+R_1)+R_2 \\
			F_2= (1+K_3^{(-)}(t,x,D_x))^{-1}e^{-\mu(t) (\ln(1+\Phi(x) \la D_x \rak))^{\g} } (1+K_3^{(+)}(t,x,D_x)) F_1
		\end{rcases}
	\end{equation}
	and the operator $R_2$ is such that
	\begin{linenomath*}
		\[
		\begin{aligned}
			\sigma(R_2)(t,x,\xi)  &\in G^{0,0}\{1,1;\g,\delta_1\}_{int,6} (\Phi, g_{\Phi,k}) + 	G^{0,0}\{1,0,1,1;\g,\delta\}_{ext,1} (\Phi, g_{\Phi,k}),
		\end{aligned}
		\]
	\end{linenomath*}
	in other words $t^{1-\varepsilon}\sigma(R_2) \in {C([0,T];G^{0,0}(\Phi;g^{(1,\delta),(1-\delta,0)}_{\Phi,k})}.$
	From (\ref{lb1}), (\ref{lb2}), (\ref{R1}) and noting the fact that the operator $t^{1-\varepsilon}R_2$ is uniformly bounded in $L^2(\mathbb{R}^n)$ for $0 \leq t \leq T$, it follows that
	\begin{equation}
		\label{K}
		2 \Re \langle \mathcal{K}V_2,V_2 \rangle_{L^2} \geq -\frac{C}{t^{1-\varepsilon}}\la V_2,V_2 \ra_{L^2}, \quad C>0,
	\end{equation}
	where $\mathcal{K} = \mathcal{D}+(\psi I+A)+(\kappa_1t^{-1+\varepsilon} (\ln (1+\Phi(x) \la 	D_x\ran))^\g+R_1)+R_2$.
	From (\ref{last}) and (\ref{K}), we have
	\begin{linenomath*}
		\[
		\partial_t \Vert V_2 \Vert_{L^2}^2 \leq C(t^{-1+\varepsilon} \Vert V_2 \Vert_{L^2}^2 + \Vert F_2 \Vert_{L^2}^2).
		\]
	\end{linenomath*}
	We apply Gronwall's lemma to obtain
	\begin{linenomath*}
		\[
		\Vert  V_2(t,\cdot)\Vert _{L^2}^2\leq e^{Ct^\varepsilon/\varepsilon} \left( \Vert  V_2(0,\cdot)\Vert ^2_{L^2} + \int_{0}^{t}\Vert F_2(\tau,\cdot)\Vert_{L^2}d\tau \right),
		\]
	\end{linenomath*}
	for $t \in [0,T]$. In other words, from (\ref{V1}),
	\begin{linenomath*}
		\begin{equation}\label{eq21}
			\Vert {U}(t,\cdot)\Vert_{\Phi,k;s,-\Lambda(t),\g}^2\leq C_{\varepsilon} \left(\Vert {U}(0,\cdot)\Vert ^2_{\Phi,k;s} + \int_{0}^{t}\Vert F(\tau,\cdot)\Vert_{\Phi,k;s,-\Lambda(\tau),\gamma}\;d\tau \right).
		\end{equation}
	\end{linenomath*}
	Here $C_{\varepsilon} = 4^{\alpha(T)+1}e^{CT^\varepsilon/\varepsilon}$.
	Returning to our original solution $u=u(t,x)$, we obtain 
	\begin{linenomath*}
		\[
		\begin{aligned}
			\sum_{j=0}^{1} \Vert \partial_t^j &u(t,\cdot) \Vert_{\Phi,k;s+(1-j)e,-\Lambda(t),\gamma} \\ 
			&\leq  C_{\varepsilon} \Bigg(\sum_{j=1}^{2} \Vert f_j\Vert_{\Phi,k;s+(2-j)e} 
			+ \int_{0}^{t}\Vert f(\tau,\cdot)\Vert_{\Phi,k;s,-\Lambda(\tau),\gamma}\;d\tau\Bigg),
		\end{aligned}
		\]
	\end{linenomath*}
	for $0 \leq t \leq T$. This means that the original problem (\ref{eq1}) is well-posed for $u=u(t,x)$, with 
	\begin{linenomath*}
		\[
		u \in C([0,T];\mathcal{H}^{s+e,-\Lambda(t),\g}_{\Phi,k}) \cap C^{1}([0,T];\mathcal{H}^{s,-\Lambda(t),\g}_{\Phi,k}).
		\]
	\end{linenomath*}
	
	This shows that when (\ref{bl1}) is satisfied one has atmost an arbitrarily small loss. When (\ref{bl2}) is satisfied instead of (\ref{bl1}), the proof follows in similar lines. Note that the condition (\ref{bl2}) suggests that the the coefficients are bounded in $t$. This implies that the majorizing functions in (\ref{tpsi}) are zero order symbols in both $x$ and $\xi$ as $\delta_3 \in [0,1)$ in (\ref{bl2}). Implying that the elliptic operators involoved in the changes of variable as in (\ref{c1}) and (\ref{c2}) are zero order pseudodifferential operators and hence one can take $\Lambda(t)= 0$ in the above discussion. This suggests that there is no loss in regularity.
	Thus concludes the proof.
	
	\section{Cone Condition}\label{cone}
	In the following, we prove the existence of cone of dependence for the Cauchy problem (\ref{eq1}).	
	We note here that the $L^1$ integrability of the sublogarithmic singularity plays a crucial role in arriving at the finite propagation speed. The implications of the discussion in \cite[Section 2.3 \& 2.5]{JR} to the global setting suggest that if the Cauchy data in (\ref{eq1}) is such that $f \equiv 0$ and $f_1,f_2$ are supported in the ball $\vert x \vert \leq R$, then the solution to Cauchy problem (\ref{eq1}) is supported in the ball $\vert x \vert \leq R+c^* \o \t t$. Note that the support of the solution increases as $|x|$ increases since $\o$ is monotone increasing function of $|x|.$ Recall that $\t=(\ln(1+1/t))^{\g}.$ The quantity $t\t$ is bounded in $[0,T]$. The constant $c^*$ is such that the quantity $c^* \o \t$ dominates the characteristic roots, i.e.,
	\begin{equation}
		\label{speed}
		c^*= \sup\Big\{\sqrt{a(t,x,\xi)}\o^{-1} \t^{-1}:(t,x,\xi) \in[0,T] \times \R^n_x \times \R^n_\xi,\:|\xi|=1\Big\}.
	\end{equation}
	
	In the following we prove the cone condition for the Cauchy problem $(\ref{eq1})$ as in \cite[Section 3.11]{yag}. Let $K(x^0,t^0)$ denote the cone with the vertex $(x^0,t^0)$:
	\begin{linenomath*}
		\[
		K(x^0,t^0)= \{(t,x) \in [0,T] \times \R^n : |x-x^0| \leq c^*\o \theta(t^0-t)(t^0-t)\}.
		\]
	\end{linenomath*}
	Observe that the slope of the cone is anisotropic, that is, it varies with both $x$ and $t$.
	
	\begin{prop}
		The Cauchy problem (\ref{eq1}) has a cone dependence, that is, if
		\begin{equation}\label{cone1}
			f\big|_{K(x^0,t^0)}=0, \quad f_i\big|_{K(x^0,t^0) \cap \{t=0\}}=0, \; i=1, 2,
		\end{equation}
		then
		\begin{equation}\label{cone2}
			u\big|_{K(x^0,t^0)}=0.
		\end{equation}
	\end{prop}
	\begin{proof}
		Consider $t^0>0$, $\gamma>0$ and assume that  (\ref{cone1}) holds. We define a set of operators $P_\varepsilon(t,x,\partial_t,D_x), 0 \leq \varepsilon \leq \varepsilon_0$ by means of the operator $P(t,x,\partial_t,D_x)$ in (\ref{eq1}) as follows
		\begin{linenomath*}
			\[
			P_\varepsilon(t,x,\partial_t,D_x) = P(t+\varepsilon,x,\partial_t,D_x), \: t \in [0,T-\varepsilon_0], x \in \R^n,
			\]
		\end{linenomath*}
		and $\varepsilon_0 < T-t^0$, for a fixed and sufficiently small $\varepsilon_0$. For these operators we consider Cauchy problems
		\begin{linenomath*}
			\begin{alignat*}{2}
				P_\varepsilon v_\varepsilon & =f,  &&  t \in [0,T-\varepsilon_0], \; x \in \R^n,\\
				\partial_t^{k-1}v_\varepsilon(0,x)& =f_k(x),\qquad && k=1,2.
			\end{alignat*}
		\end{linenomath*}
		Note that $v_\varepsilon(t,x)=0$ in $K(x^0,t^0)$ and $v_\varepsilon$ satisfies an a priori estimate (\ref{est2}) for all $t \in[0,T-\varepsilon_0]$. Further, we have 
		\begin{linenomath*}
			\begin{alignat*}{2}
				P_{\varepsilon_1} (v_{\varepsilon_1}-v_{\varepsilon_2}) & = (P_{\varepsilon_2}-P_{\varepsilon_1})v_{\varepsilon_2},\qquad  &&  t \in [0,T-\varepsilon_0], \; x \in \R^n,\\
				\partial_t^{k-1}(v_{\varepsilon_1}-v_{\varepsilon_2})(0,x)& = 0,\qquad && k=1,2.
			\end{alignat*}
		\end{linenomath*}
		Since our operator is of second order, for the sake of simplicity we denote $b_{j}(t,x)$, the coefficients of lower order terms, as $a_{0,j}(t,x), 1 \leq j \leq n,$ and $b_{n+1}(t,x)$ as $a_{0,0}(t,x)$. Let $a_{i,0}(t,x) =0, \; 1 \leq i \leq n.$ Substituting $s-e$ for $s$ in the a priori estimate, we obtain
		\begin{equation}\label{cone3}
			\begin{aligned}
				&\sum_{j=0}^{1} \Vert  \partial_t^j(v_{\varepsilon_1}-v_{\varepsilon_2})(t,\cdot) \Vert_{\Phi,k;s-(j+\Lambda(t))e} \\
				&\leq C \int_{0}^{t}\Vert (P_{\varepsilon_2}-P_{\varepsilon_1})v_{\varepsilon_2}(\tau,\cdot)\Vert_{\Phi,k;s-e-\Lambda(\tau)e}\;d\tau\\
				&\leq C \int_{0}^{t} \sum_{i,j=0 }^{n}\Vert (a_{i,j}(\tau+\varepsilon_1,x) - a_{i,j}(\tau+\varepsilon_2,x))D_{ij} v_{\varepsilon_2}(\tau,\cdot)\Vert_{\Phi,k;s-e-\Lambda(\tau)e}\;d\tau,
			\end{aligned}
		\end{equation}
		where $D_{00}=I, D_{i0}=0, i\neq 0, D_{0j}=\partial_{x_j},j \neq 0$ and $D_{ij}=\partial_{x_i} \partial_{x_j}, i,j \neq 0$. Using the Taylor series approximation in $\tau$ variable, we have
		\begin{linenomath*}
			\begin{align*}
				|a_{i,j}(\tau+\varepsilon_1,x) - a_{i,j}(\tau+\varepsilon_2,x)| &= \Big|\int_{\tau+\varepsilon_2}^{\tau+\varepsilon_1} (\partial_ta_{i,j})(r,x)dr \Big|\\
				&\leq \o^{2} \Big|\int_{\tau+\varepsilon_2}^{\tau+\varepsilon_1}\frac{\t^{1-\frac{1}{\g}}}{r}dr\Big|\\
				&\leq \o^{2}|E(\tau,\varepsilon_1,\varepsilon_2)|,
			\end{align*}
		\end{linenomath*}
		where
		\begin{linenomath*}
			\[
			E(\tau,\varepsilon_1,\varepsilon_2) =		
			\left(\ln \Bigg(1+\frac{\varepsilon_1-\varepsilon_2}{\tau+\varepsilon_2}\Bigg)\right)^\g . 
			\]
		\end{linenomath*}
		Note that $\o \lesssim \P$ and $E(\tau,\varepsilon,\varepsilon)=0$.
		Then right-hand side of the inequality in (\ref{cone3}) is dominated by
		\begin{linenomath*}
			\begin{equation*}\label{cone4}
				C \int_{0}^{t} |E(\tau,\varepsilon_1,\varepsilon_2)| \Vert  v_{\varepsilon_2}(\tau,\cdot)\Vert_{\Phi,k;s+(1-\Lambda(\tau))e}\;d\tau,
			\end{equation*}
		\end{linenomath*}
		where $C$ is independent of $\varepsilon$. By definition, $E$ is $L_1$-integrable in $\tau$.
		
		The sequence $v_{\varepsilon_k}$, $k=1,2,\dots$ corresponding to the sequence $\varepsilon_k \to 0$ is in the space
		\begin{linenomath*}
			\[
			C\Big([0,T^*];H^{s-\nu e,}_{\Phi,k}\Big) \bigcap C^{1}\Big([0,T^*];H^{s-e-\nu e,}_{\Phi,k}\Big), \quad T^*>0,
			\]
		\end{linenomath*}
		for an arbitrarily small $\nu>0$ and $u=\lim\limits_{k\to\infty}v_{\varepsilon_k}$ in the above space and hence, in $\mathcal{D}'(K(x^0,t^0))$. In particular,
		\begin{linenomath*}
			\[
			\la u,\varphi\ra = \lim\limits_{k\to\infty} \la v_{\varepsilon_k}, \varphi \ra =0,\; \forall \varphi \in \mathcal{D}(K(x^0,t^0))
			\]
		\end{linenomath*}
		gives (\ref{cone2}) and completes the theorem. 
	\end{proof}
	\newpage
	
	\hspace{-0.6cm}{\bf Funding:} The first author is funded by the University Grants Commission, Government of India, under its JRF and SRF schemes.\\\\	
	{\bf Data Availability Statement:} Data sharing not applicable to this article as no datasets were generated or analysed during the current study.
	
	\section*{Declarations}
	{\bf Conflicts of Interests:} The authors declare that they have no competing interests.\\\\

\end{document}